\documentclass[11pt,twoside,reqno]{amsart}
\usepackage{amsmath, amsfonts, amsthm, amssymb,amscd, graphicx, amscd}
\usepackage{float,epsf}
\usepackage[english]{babel}
\usepackage{enumerate}
\usepackage{tikz}
\usepackage{bm}

\usepackage{srcltx}

\usepackage{geometry}
\usepackage{verbatim}
\usepackage{mathrsfs}
\usepackage{hyperref}

\newtheorem{theorem}{Theorem}[section]

\newtheorem{definition}[theorem]{Definition}

\newtheorem{proposition}[theorem]{Proposition}
\newtheorem{corollary}[theorem]{Corollary}
\newtheorem{remark}[theorem]{Remark}
\newtheorem{example}[theorem]{Example}

\newtheorem{question}[theorem]{Question}

\numberwithin{equation}{section}
\numberwithin{figure}{section}

\renewcommand{\div}{\mathrm{div}} 

\def\intave#1{\int_{#1}\hbox{\llap{$\raise2.3pt\hbox{\vrule
height.9pt width7pt}\phantom{\scriptstyle{#1}}\mkern-2mu$}}}

\begin{document}
\title{On smooth interior approximation of Sets of Finite Perimeter}

\author{Changfeng Gui}
\address{Changfeng Gui, The University of Texas at San Antonio, TX, USA.
}
\email{changfeng.gui@utsa.edu}
\author{Yeyao Hu}
\address{Yeyao Hu, School of Mathematics and Statistics, HNP-LAMA, Central South University, Changsha, Hunan 410083, P. R. China.
}
\email{huyeyao@gmail.com}

\author{Qinfeng Li}
\address{Qinfeng Li, School of Mathematics, Hunan University, Changsha, Hunan, China.
}
\email{liqinfeng1989@gmail.com}

\maketitle
\begin{abstract}
In this paper, we prove that for any bounded set of finite perimeter $\Omega \subset \mathbb{R}^n$, we can choose smooth sets $E_k \Subset \Omega$ such that $E_k \rightarrow \Omega$ in $L^1$ and \begin{align}
    \label{moregeneralapproximation}
\limsup_{i \rightarrow \infty} P(E_i) \le P(\Omega)+C_1(n) \mathscr{H}^{n-1}(\partial \Omega \cap \Omega^1).
\end{align}In the above $\Omega^1$ is the measure-theoretic interior of $\Omega$, $P(\cdot)$ denotes the perimeter functional on sets, and $C_1(n)$ is a dimensional constant. 

Conversely, we prove that for any sets $E_k \Subset \Omega$ satisfying $E_k \rightarrow \Omega$ in $L^1$, there exists a dimensional constant $C_2(n)$ such that the following inequality holds:
\begin{align}
\label{gap}
   \liminf_{k \rightarrow \infty} P(E_k) \ge P(\Omega)+ C_2(n) \mathscr{H}^{n-1}(\partial \Omega \cap \Omega^1).
\end{align}
In particular, these results imply that for a bounded set $\Omega$ of finite perimeter,\begin{align}
\label{char*}
    \mathscr{H}^{n-1}(\partial \Omega \cap \Omega^1)=0
\end{align}
holds if and only if there exists a sequence of smooth sets $E_k$ such that $E_k \Subset \Omega$, $E_k \rightarrow \Omega$ in $L^1$ and $P(E_k) \rightarrow P(\Omega)$. 

\end{abstract}



\section{Introduction}
\subsection{Motivation}
Let $\Omega$ be a bounded set of finite perimeter in $\mathbb{R}^n$. It is well known that for such $\Omega$, there exists a sequence of smooth sets $E_k$ such that $E_k \rightarrow \Omega$ in $L^1$ and $P(E_k) \rightarrow P(\Omega)$ as $k \rightarrow \infty$, see for example in \cite[Theorem 13.8]{Maggi}.  However, in general such smooth sets $E_k$ cannot be chosen to be compactly contained in $\Omega$, as indicated in \cite{Schmidt} for domains containing interior fractures or cracks:
\begin{example}
\label{modelexample}
Let $R$ be the open rectangle $(-1,1) \times (-1,1)$ in $\mathbb{R}^2$ and $S=(-1,1)\times \{0\}$. Then set $\Omega:=R\setminus S=R_1 \cup R_2$, where $R_1$ is the upper open rectangle $(-1,1)\times (0,1)$ and $R_2$ is the lower open rectangle $(-1,1) \times (-1,0)$. Then for any sets $E_k \Subset \Omega$ and $E_k \rightarrow \Omega$ in $L^1$, we have $P(E_k)=P(E_k \cap R_1)+P(E_k \cap R_2)$, where we have used the fact $E_k \Subset \Omega$. Hence by lower semicontinuity of perimeter functionals we have
\begin{align}
\label{model}
    \liminf_{k \rightarrow \infty} P(E_k) \ge P(R_1)+P(R_2)=P(\Omega)+2 \mathscr{H}^{n-1}(S)>P(\Omega).
\end{align}
\end{example}

This paper was originally motivated by the following question concerning the one-sided approximation of set of finite perimeter by smooth sets strictly from within:

\begin{question}
\label{fund2}
What is the optimal condition on a  bounded set $\Omega\subset \mathbb{R}^n$ of finite perimeter such that there exist smooth sets $E_k \Subset \Omega, E_k \rightarrow \Omega$ in $L^1$ and $P(E_k)\rightarrow P(\Omega)$?
\end{question}

The answer to the question is not only of self interest in the theory of set of finite perimeter, but also connected to many other problems such as the study of normal traces of bounded divergence measure fields from the view point of interior approximation (see \cite{CCT}, \cite{CLT}, \cite{CTZ} and many other references therein), the prescribed mean curvature problem (see \cite{LS} and references therein), the equations of $1$-Laplacian and minimal surface type (see \cite{SS16}) and so on.

In this paper, we have found the necessary and sufficient condition for bounded set $\Omega$ to have an interior approximation sequence as described in Question \ref{fund2}. In order to state previous results and our results more precisely, let us first introduce some notations, most of which are from standard textbooks such as \cite{afp} and \cite{Maggi} : 

For a measurable set $\Omega \subset \mathbb{R}^n$, we use $|\Omega|$ to denote the Lebesgue measure of $\Omega$. We use $\Omega^1$ to denote the measure theoretic interior of $\Omega$ (i.e., the set of points in $\mathbb{R}^n$ where $\Omega$ has density $1$ with respect to Lebesgue measure), and we use $\Omega^0$ to denote the the measure theoretic exterior of $\Omega$  (i.e., the set of points in $\mathbb{R}^n$ where $\Omega$ has density $0$ with respect to Lebesgue measure). We use $\mathscr{H}^{n-1}(\cdot)$ to denote the $(n-1)$-Hausdorff measure of a set.  The reduced boundary of $\Omega$ is denoted as $\partial^* \Omega$, and the measure-theoretic boundary of $\Omega$ is denoted as $\partial^m \Omega$. We say $E_i \rightarrow \Omega$ in $L^1$, if $|E_i \Delta \Omega| \rightarrow 0$.
\hfill\\

\subsection{Previous results} There have been several important results related to interior approximation of sets of finite perimeter in literature.

First, in Chen-Torres-Ziemer\cite{CTZ}, the authors prove that for any set of finite perimeter, there exists smooth sets approximating the set essentially from the measure-theoretic interior with respect to any Radon measure which is absolutely continuous with respect to $\mathscr{H}^{n-1}$. This approximation works for general set of finite perimeter, but the approximating sequence of sets are not compactly contained in $\Omega$ as described in Question \ref{fund2}. We refer to \cite{CTZ} for precise statement and applications, and see also \cite{MG17} for a new proof which also fills the gap in \cite{CTZ}.

Also in \cite{CTZ}, the authors give a sufficient condition for $\Omega$ to have an interior approximation sequence of sets which are compactly contained in $\Omega$ as in Question \ref{fund2}:
\begin{theorem}(\cite[Theorem 8.2]{CTZ})
\label{dens}
If there exist $\delta>0, r_0>0$ such that
\begin{align}
\label{density}
    \frac{|B_r(x) \cap \Omega|}{|B_r(x)|}< 1-\delta, \quad \forall x \in \partial \Omega, \, \forall\, 0<r\le r_0
\end{align}
then there exists a smooth sequence $E_k \Subset \Omega$, $E_k \rightarrow \Omega$ in $L^1$ and $P(E_k) \rightarrow P(\Omega)$ as $k \rightarrow \infty$.
\end{theorem}

The approximation sets $E_k$ above are actually constructed as superlevel sets of the mollified characteristic function of $\Omega$. The idea is that \eqref{density} implies that for $t \in (C(\delta,n), 1)$, $\{\chi_{\Omega} *\rho_{\epsilon}>t\} \Subset \Omega$ for $\epsilon>0$ small enough, where $\rho_{\epsilon}$ is the standard mollifier. See \cite{CTZ} for more details.

Next, in \cite[Theorem 1.1]{Schmidt}, Schmidt made the following observation: if 
\begin{align}
\label{xiang1}
    \mathscr{H}^{n-1}(\partial \Omega)=P(\Omega),
\end{align}then there exist smooth sets $E_k \Subset \Omega$ such that $E_k \rightarrow \Omega$ in $L^1$ and $P(E_k) \rightarrow P(\Omega)$. The proof utilizes the rectifiability of $\partial \Omega$, since \eqref{xiang1} implies $\partial \Omega$ is $\mathscr{H}^{n-1}$ equivalent to its reduced boundary, which is rectifiable due to De Giorgi's structure theorem.

We note that there are sets satisfying \eqref{density} but not satisfying \eqref{xiang1}. For example, in \cite{BGM09} the authors construct a pseudoconvex set such that its boundary has positive Lebesgure measure, and thus \eqref{xiang1} fails. However, it is known that all pseudoconvex sets satisfy \eqref{density}, see for example \cite[Corollary 7.17]{LT} for a proof of such property. On the other hand, piecewise Lipschitz domains with finite inward cusps satisfy \eqref{xiang1}, but \eqref{density} fails for such domains. Hence neither \eqref{density} nor \eqref{xiang1} is the necessary condition for domains having interior approximation sequence described in Question \ref{fund2}.

Another notable result related to smooth interior approximation is proved in \cite{CLT} by Chen, Torres and the third author of the paper as follows:
\begin{theorem}(\cite[Theorem 3.1]{CLT})
\label{linshi}
Let $\Omega$ be a bounded set of finite perimeter in $\mathbb{R}^n$. Then there exist smooth sets $E_k \Subset \Omega,$ such that $E_k \rightarrow \Omega$ in $L^1$ and $\sup_kP(E_k) <\infty$, if and only if 
\begin{align}
\label{neato}
    \mathscr{H}^{n-1}(\partial \Omega \setminus \Omega^0)<\infty,
\end{align}
\end{theorem}
We also remark that for $\Omega$ being a set of finite perimeter, \eqref{neato} is equivalent to \begin{align}
\label{finite}
    \mathscr{H}^{n-1}(\partial \Omega \cap \Omega^1)<\infty
\end{align}due to Federer's Theorem, see Theorem \ref{Fed} in the next section.

The importance of Theorem \ref{linshi} is that it is the first result to characterize domains which can be approximated by smooth sets of uniformly bounded perimeters compactly contained within. It is used to prove that domains satisfying \eqref{neato} are extension domains for bounded divergence measure fields, thus generalized Gauss-Green formula for bounded divergence measure fields can be established in such domains. The proof of the existence of the approximating sequence of sets under the hypothesis \eqref{neato} is based on a fine covering of $\partial \Omega \setminus \Omega^0$.

\hfill\\

\subsection{Our Results}
In this paper, we first generalize the results in \cite{Schmidt}. Instead of assuming condition \eqref{xiang1}, which is by Theorem \ref{Fed} equivalent to
\begin{align}
    \mathscr{H}^{n-1}(\partial \Omega \cap \Omega^1)+\mathscr{H}^{n-1}(\partial \Omega \cap \Omega^0)=0,
\end{align}to construct the desirable smooth interior approximation sequence strictly from inside, we only require that the set of the measure-theoretic interior points of $\Omega$ on $\partial \Omega$ is $\mathscr{H}^{n-1}$-negligible, that is, \begin{align}
\label{char}
    \mathscr{H}^{n-1}(\partial \Omega \cap \Omega^1)=0.
\end{align}We have the following proposition.
\begin{proposition}
\label{sufw}
Let $\Omega$ be a bounded set of finite perimeter in $\mathbb{R}^n$. If \eqref{char} is satisfied, then there exist smooth sets $E_k \Subset \Omega$ such that $E_k \rightarrow \Omega$ in $L^1$ and $P(E_k) \rightarrow P(\Omega)$. 
\end{proposition}

One can see that \eqref{char} is equivalent to\begin{align}
\label{zaozuo}
    \limsup_{r \rightarrow 0} \frac{|B_r(x) \cap \Omega|}{|B_r(x)|}< 1, \quad \mathscr{H}^{n-1}-a.e.\, x \in \partial \Omega.
\end{align}This weakens the condition \eqref{density} by removing the uniform constant gap $\delta$ and by removing the uniform distance $r_0$ near $\partial \Omega$, and hence Proposition \ref{sufw} also generalizes Theorem \ref{dens}.

To prove Proposition \ref{sufw}, it suffices to prove the following more general theorem from which Proposition \ref{sufw} follows easily.

\begin{theorem}
\label{chengxinxingyi}
For any $\Omega$ being a bounded set of finite perimeter in $\mathbb{R}^n$, we can choose a sequence of smooth sets $E_i \Subset \Omega$ such that $E_i \rightarrow \Omega$ in $L^1$ and \begin{align}
    \label{moregeneralapproximation}
\limsup_{i \rightarrow \infty} P(E_i) \le P(\Omega)+C(n) \mathscr{H}^{n-1}(\partial \Omega \cap \Omega^1).    
\end{align}
\end{theorem}

The proof of Theorem \ref{chengxinxingyi} uses some argument in \cite{Schmidt} and also partially utilizes the idea in the proof of Theorem \ref{linshi} in \cite{CLT}, but refines the analysis of perimeters inside $\Omega$ of the covering balls centered at $\partial \Omega \cap \Omega^0$.
\hfill\\

Next, we also prove that the condition \eqref{char} is sharp, by showing that it is also the necessary condition for a bounded set $\Omega$ to have an interior approximation sequence in the sense of Question \ref{fund2}. In fact, we can prove the following stronger result, which estimates how much gap between the limit of perimeters of any $L^1$ approximation sequence compactly contained in $\Omega$ and $P(\Omega)$ could be. 
\begin{theorem}
\label{nec}
Let $\Omega$ be a bounded set of finite perimeter. For any $E_i \Subset \Omega$ with $E_i \rightarrow \Omega$ in $L^1$, we have
\begin{align}
\label{wa}
   \liminf_{i \rightarrow \infty} P(E_i) \ge P(\Omega)+C(n)\mathscr{H}^{n-1}(\partial \Omega \cap \Omega^1),
\end{align}where $C(n)$ is a constant depending only on dimension $n$.
\end{theorem}

Clearly if $E_i \Subset \Omega$, $E_i \rightarrow \Omega$ in $L^1$ and $P(E_i)\rightarrow P(\Omega)$ then \eqref{wa} automatically implies \eqref{char}. 

Hence the combination of Proposition \ref{sufw} and Theorem \ref{nec} gives the following characterization result for sets of finite perimeters admitting strict interior approximation sequence described in Question \ref{fund2}:
\begin{theorem}
\label{suffnece}
Let $\Omega$ be a bounded set of finite perimeter, then \eqref{char} holds if and only if there exist a sequence of smooth sets $E_k$ compactly contained in $\Omega$ such that $E_k \rightarrow \Omega$ in $L^1$ and $P(E_k) \rightarrow P(\Omega)$ as $k \rightarrow \infty$.
\end{theorem}
\hfill\\

\subsection*{Some more remarks}
We first remark that we can also consider strict exterior approximation of sets of finite perimeter, and similar results hold. Indeed, for bounded set $\Omega$, we may assume $\Omega \Subset B_R$ for some large $R$, then by considering $B_R \setminus \Omega$, previous results imply the following corollaries:

\begin{corollary}
\label{exterior}
Let $\Omega$ be a bounded set of finite perimeter, then
\begin{align}
    \label{charout}
\mathscr{H}^{n-1}(\partial \Omega \cap \Omega^0)=0  
\end{align}
holds if and only if there exist a sequence of smooth sets $E_k \Supset \Omega$ such that $E_k \rightarrow \Omega$ in $L^1$ and $P(E_k) \rightarrow P(\Omega)$ as $k \rightarrow \infty$.
\end{corollary}

\begin{corollary}
For any bounded set of finite perimeter $\Omega \subset \mathbb{R}^n$, we can choose smooth sequence $E_k \Supset \Omega$ such that $E_k \rightarrow \Omega$ in $L^1$ and \begin{align*}
\limsup_{i \rightarrow \infty} P(E_i) \le P(\Omega)+C_1(n) \mathscr{H}^{n-1}(\partial \Omega \cap \Omega^0)
\end{align*}for some dimensional constant $C_1(n)$.
Conversely, for any sets $E_k \Supset \Omega$ satisfying $E_k \rightarrow \Omega$ in $L^1$, there exists a dimensional constant $C_2(n)$ such that the following inequality holds:
\begin{align}
\label{gap}
   \liminf_{k \rightarrow \infty} P(E_k) \ge P(\Omega)+ C_2(n) \mathscr{H}^{n-1}(\partial \Omega \cap \Omega^0).
\end{align}
\end{corollary}
 
Next, we remark that Proposition \ref{sufw} fully generates \cite[Theorem 1.1]{Schmidt}. Indeed, \cite[Theorem 1.1]{Schmidt} states that under the hypothesis $\mathscr{H}^{n-1}(\partial \Omega)=P(\Omega)$, there exists sets $E_{\epsilon}$ such that $E_{\epsilon} \Subset \Omega$, $P(E_{\epsilon})<P(\Omega)+\epsilon$ and moreover 
\begin{align}
\label{hausforffdistance}
    \Omega \setminus E_{\epsilon} \subset \mathcal{N}_{\epsilon}(\partial \Omega) \cap \mathcal{N}_{\epsilon}(\partial E_{\epsilon}),
\end{align}where $\mathcal{N}_{\epsilon}(\cdot)$ denotes the $\epsilon$-neighborhood of a set. We claim that under the weaker assumption \eqref{char}, the above conclusions also hold. Indeed, it suffices to prove \eqref{hausforffdistance} for the sets constructed in our proof, which is clearly true since the diameters of the covering sets we chose can be arbitrarily smaller.

Moreover, we note that Theorem \ref{chengxinxingyi} and Theorem \ref{nec} also improve Theorem \ref{linshi}, which gives the characterization of sets having strict interior approximation sequence of uniformly bounded perimeters.  Indeed, if $\mathscr{H}^{n-1}(\partial \Omega \setminus \Omega^0)<\infty$, then by Theorem \ref{chengxinxingyi} we can construct smooth sets $E_k$ satisfying \eqref{moregeneralapproximation}, and hence from \eqref{moregeneralapproximation} these sets  must have uniformly bounded perimeters. Conversely, if such approximating sets exist, then immediately from \eqref{wa} we have $\mathscr{H}^{n-1}(\partial \Omega \setminus \Omega^0)=P(\Omega)+\mathscr{H}^{n-1}(\partial \Omega \cap \Omega^1)<\infty.$ 

In addition, we remark that these theoretical results are potentially applicable to some other problems related to rough domains. For example, it can be shown that strict interior approximation of BV functions studied in \cite{Schmidt} also holds for domains satisfying \eqref{char} by the construction in the proof of Theorem \ref{chengxinxingyi}. Another example is that under the condition \eqref{char}, we can replace the approximation sequence of sets $E_k$ of uniformly bounded perimeters in the proof of \cite[Theorem 5.2]{CLT} by sets described in Question \ref{fund2} , thus the same argument in \cite{CLT} can prove that the trace of any bounded divergence measure field $F$ is not only $L^{\infty}$, but also its $L^{\infty}$ norm is bounded by $\Vert F\Vert _{\infty}$. This improves the result in \cite[Theorem 5.2]{CLT} under the stronger assumption \eqref{char} from the interior approximation perspective. (A different proof can be found in \cite{CT05} or \cite{CTZ} using measure theory.) We also remark that the advantage of the existence of smooth interior approximation sequence $E_i$ also lies in the fact that for any bounded continuous function $g$, $\int_{\partial E_i} g d\mathscr{H}^{n-1} \rightarrow \int_{\partial^* \Omega} g d\mathscr{H}^{n-1}$. This property could potentially to be used in finding weak solutions to some equations on rough domains satisfying \eqref{char} through interior approximation argument. 

Finally, our upcoming paper \cite{GHL2} will apply Theorem \ref{chengxinxingyi} to prove sharp estimates for the total variation of distributional gradient of extended BV functions and divergence measure fields, and we will also apply this theorem to prove approximation results for trace of bounded BV functions.


\hfill\\

\subsection{Sketch of proofs}
Since the proofs of our main results Theorem \ref{chengxinxingyi} and Theorem \ref{nec} are somewhat technical, we would like to outline the ideas behind those proofs to convince the reader.

The proof of Theorem \ref{chengxinxingyi} partially uses the argument in \cite{Schmidt} and \cite{CLT}. In \cite{Schmidt}, under the assumption \eqref{xiang1}, in order to show the existence of strict interior approximation sets, it suffices to find a nice covering of $\partial^* \Omega$. The idea in \cite{Schmidt} is that one can cover $\partial^* \Omega$ by flat cylinders, due to De Giorgi's structure Theorem (see for example \cite{Maggi}) for the reduced boundary of set of finite perimeter. Because the cylinders are quite flat, the total perimeter inside $\Omega$ of the cylinders is approximately the $\mathscr{H}^{n-1}$ measure of the reduced boundary of $\Omega$, which is the perimeter of $\Omega$. In our Theorem \ref{chengxinxingyi}, we need to additionally deal with $\partial \Omega\cap \Omega^0$, $\partial \Omega\cap \Omega^1$, and the rest of $\partial \Omega$ which is $\mathscr{H}^{n-1}$ negligible set. The essence of the proof is to delicately choose a desirable covering of $\partial \Omega \cap \Omega^0$. The idea is that although the $\mathscr{H}^{n-1}$ measure, or even the Lebesgue measure of $\partial \Omega \cap \Omega^0$ could be positive, see for example \cite[Example 8.4]{LT} or \cite{BGM09}, we can still cover $\partial \Omega \cap \Omega^0$ by sufficiently small balls such that the total perimeter of the balls contributed inside $\Omega$ is bounded by (up to a universal constant factor) the total perimeter of $\Omega$ inside these small balls, which is very small because the perimeter of $\Omega$ is concentrated on $\partial^* \Omega$. The details of how to choose such covering balls will be carefully explained in the proof. Hence so far we have chosen  flat cylinders covering $\partial^* \Omega$ and balls covering $\partial \Omega \cap \Omega^0$ and such that the total perimeter of the covering sets contributed inside $\Omega$ is approximately the $\mathscr{H}^{n-1}$ measure of $\partial^* \Omega$, then we can cover the rest of the boundary, which is $\mathscr{H}^{n-1}$-equivalent to $\partial \Omega \cap \Omega^1$, by small balls whose total perimeters are bounded above by $C(n) \mathscr{H}^{n-1}(\partial \Omega \cap \Omega^1)$ and total volumes are very small. Let $Q$ be the union of these covering sets. Hence the perimeter of $\Omega \setminus Q$, which is roughly the perimeter of $Q$ inside $\Omega$, is approximately bounded by the perimeter of $\Omega$ plus $C(n)\mathscr{H}^{n-1}(\partial \Omega \cap \Omega^1)$ with small error depending on the size of the covering sets. As the covering balls and flat cylinders are chosen to be smaller and smaller, we thus obtain a desired sequence $E_i$ as in Theorem \ref{chengxinxingyi}.

The idea of the proof of Theorem \ref{nec} is roughly as follows: Given such an approximation sequence $E_i$, one can observe that inside a suitably chosen small neighborhood of $\partial  \Omega \cap \Omega^1$, the perimeter of $E_i$ is roughly bounded below by $\mathscr{H}^{n-1}(\partial \Omega \cap \Omega^1)$ up to a dimensional constant. This can be proved by relative isoperimetric inequality and a covering argument. Such a neighborhood of $\partial  \Omega \cap \Omega^1$ can be so small that the perimeter of $\Omega$ inside the closure of the neighborhood is also very small, since perimeter measure is only concentrated on the reduced boundary of $\Omega$, not on $\partial \Omega \cap \Omega^1$. Therefore the perimeter of $\Omega$ mainly concentrated outside the closure of this neighborhood, and thus by lower semicontinuity of sets of finite perimeter measure, the perimeter of $E_i$ outside the neighborhood is larger than the perimeter of $\Omega$ as $i$ becomes large. Hence adding up the perimeter of $E_i$ inside and outside the above chosen neighborhood, we can see that the total perimeter of $E_i$, as $i$ becomes large, is approximately bounded below by the total perimeter of $\Omega$ plus $C(n)\mathscr{H}^{n-1}(\partial \Omega \cap \Omega^1)$ for some constant $C(n)>0$. This concludes the result in Theorem \ref{nec}. The rigorous argument can be found in the proof.
\hfill\\
\subsection*{Outline of the paper} In section 2, we recall some basic facts of the theory of sets of finite perimeter. In section 3, we give the full detailed proof of our main results Theorem \ref{chengxinxingyi} and Theorem \ref{nec}.

\section{PRELIMINARIES}
In this section we first recall some properties of sets of finite perimeter. More results can be seen from standard references by Ambrosio-Fusco-Pallara \cite{afp},  Evans \cite{Evans}, Maggi \cite{Maggi}, etc.

We start with some basic notions and definitions. First, we denote by $\mathscr{H}^{n-1}$ the $\left(n-1\right)$-dimensional
Hausdorff measure in $\mathbb{R}^{n}$, and by $\mathcal{L}^{n}$ the Lebesgue
measure in $\mathbb{R}^{n}$. We will use the notation $\mathcal{L}^{n}(E)=|E|$. The topological boundary of set $E$ is denoted as $\partial E$.  Also, we denote $B_r(x)$ as the open ball of radius $r$ and center at $x$. We let $\omega_{n}$ be the volume of the $n$-dimensional unit ball.

\begin{definition} \label{def1} Given every $\alpha\in[0,1]$ and every
$\mathcal{L}^{n}$-measurable set $E\subset\mathbb{R}^{n}$, define 
\begin{equation}
E^{\alpha}:=\{x\in\mathbb{R}^{n}\,:\, \lim_{r\rightarrow0} \frac{ |E\cap B(x,r)|}{|B(x,r)|}=\alpha\},\label{densityone}
\end{equation}
which is the set of points where $E$ has density $\alpha$ with respect to Lebesgue measure. We define the \textit{{measure-theoretic boundary}} of $E$, $\partial^{m}E$,
as 
\begin{equation}
\partial^{m}E:=\mathbb{R}^{n}\setminus(E^{0}\cup E^{1}).
\end{equation} 
\end{definition}

\begin{definition} Let $E\subset\mathbb{R}^{n}$. We say that $E$ is a \textit{set
of finite perimeter} in $\mathbb{R}^n$ if 
\begin{equation}
\label{defoffinper}
P(E):=\sup\left\{\int_{E}\div\varphi dx:\varphi\in C_{c}^{1}(\mathbb{R}^n),\Vert{\varphi}\Vert_{\infty}\leq1 \right \}<\infty.
\end{equation}
Note that \eqref{defoffinper} implies that the distributional gradient $D\chi_E$ is a finite vector valued measure in $\mathbb{R}^n$, which is also called the Gauss-Green measure. We denote its total variation as $\Vert{D\chi_E}\Vert$. 
\end{definition}
\begin{definition} \label{reducedboundary}
Let $E$ be a set of finite perimeter in $\mathbb{R}^{n}$. The \textit{{reduced
boundary}} of $E$, denoted as $\partial^{*}E$, is the set of all
points $x\in\mathbb{R}^{n}$ such that 
\begin{enumerate}
\item \ensuremath{\Vert D\chi_E \Vert(B(x,r))>0} for all \ensuremath{r>0}
;
\item The limit \ensuremath{{\bf \nu}_{E}(x):=\lim_{r\to0}\frac{D\chi_E(B(x,r))}{\Vert D\chi_E \Vert(B(x,r))}}
exists and $|{\bf \nu}_{E}(x)|=1$. 
\end{enumerate}
\end{definition}

For any Borel set $W$, we always use $P(E,W):=\Vert{D\chi_E}\Vert(W)$ to denote the perimeter of $E$ inside $W$, or the relative perimeter of $E$ in $W$.

\begin{remark} If $E$ is a set of finite perimeter in $\mathbb{R}^{n}$, then
\begin{equation}
\Vert D\chi_E\Vert=\mathcal{H}^{n-1}\lfloor_{\partial^ * E},\label{normHaus}
\end{equation}
Hence for any set $W \subset \mathbb{R}^n$,
\begin{align*}
    P(E,W)=\mathscr{H}^{n-1}(\partial^*E \cap W)
\end{align*}
\end{remark}

The following result is due to Federer, which can be found in standard textbooks such as \cite[Theorem 3.61]{afp}:

\begin{theorem} \label{Fed}
If $E$ is a set of finite perimeter in $\mathbb{R}^{n}$, then 
\begin{equation}
\partial^{*}E\subset E^{\frac{1}{2}}\subset\partial^{m}E,\quad\mathcal{H}^{n-1}(\mathbb{R}^{n}\setminus(E^{0}\cup\partial^{*}E\cup E^{1}))=0.\label{2.2a1}
\end{equation}
In particular, $E$ has density either 0 or 1/2 or 1 at $\mathcal{H}^{n-1}$-a.e.
$x\in\mathbb{R}^{n}$ and $\mathcal{H}^{n-1}$-a.e. $x\in\partial^{m}E$ belongs to $\partial^{*}E$.
\end{theorem}

We also refer to the sets $E^{0}$ and $E^{1}$ as the \textit{measure-theoretic
exterior and interior} of $E$. 


\section{Proof of Main Results}
In this section, most of time we use $C(n)$ to denote various constants depending on dimension $n$. We use $A \lesssim_{n} B$ to denote $A \le C(n) B$ for some constant $C(n)>0$.

\begin{proof}[Proof of Theorem \ref{chengxinxingyi}]
We first note that by standard smooth approximation of sets of finite perimeter, see for example \cite[Theorem 13.8]{Maggi}, it suffices for us to construct a sequence of nonsmooth sets $E_k \Subset \Omega$ such that $E_k \rightarrow \Omega$ in $L^1$ and $\limsup_kP(E_k) \le  P(\Omega)+C(n)\mathscr{H}^{n-1}(\partial \Omega \cap \Omega^1)$.

We also note that in order to prove \eqref{moregeneralapproximation}, we may WLOG assume that $\mathscr{H}^{n-1}(\partial \Omega \cap \Omega^1) <\infty$, otherwise \eqref{moregeneralapproximation} automatically holds for any sequence of sets $E_k \Subset \Omega$ with $E_k \rightarrow \Omega$ in $L^1$. The proof goes as follows.

First, for any fixed $x \in \partial \Omega \cap \Omega^0$, by definition of $\Omega^0$ there exists $\gamma_x>0$ such that when $0<r<\gamma_x$, $|\Omega \cap B_r(x)|\le \frac{1}{2}|B_r(x)|$. Since for almost every $r>0$,
\begin{align}
\label{foliation'}
    \mathscr{H}^{n-1}(\partial B_r(x) \cap \partial^* \Omega)=0,
\end{align}
applying divergence theorem for vector fields $\bm{F}(y)=\frac{y-x}{r}$ on $\Omega \cap B_r(x)$, which is a set of finite perimeter, we have the estimate
\begin{align}
    &\mathscr{H}^{n-1}(\partial B_r(x) \cap \Omega^1)\\
    \le & \frac{n|\Omega \cap B_r(x)|}{r} + P(\Omega; B_r(x)), \quad \mbox{by divergence theorem and \cite[Theorem 16.3]{Maggi}} \\
    = & n\left(\frac{|\Omega \cap B_r(x)|^{1/n}}{r}\right)|\Omega \cap B_r(x)|^{\frac{n-1}{n}}+P(\Omega; B_r(x))\\
\label{lichi}    \le & C(n)P(\Omega; B_r(x)), \quad \mbox{by \cite[Proposition 12.37]{Maggi}.}
\end{align}
Since $P(\Omega, \partial \Omega \cap \Omega^0)=\mathscr{H}^{n-1}\left(\partial ^* \Omega\cap (\partial \Omega \cap \Omega^0)\right)=0$, for any $0<\epsilon<1$ we can choose an open set $U \supset \partial \Omega \cap \Omega^0$ such that 
\begin{align}
\label{short}
    P(\Omega, U)<\epsilon.
\end{align}
Note that for any $x \in \partial \Omega \cap \Omega^0$, we can choose $0<r_x<\gamma_x$ so small such that $B_{r_x} \subset U$ and \eqref{foliation'} holds for $r=r_x$, hence by \eqref{lichi} we have  
\begin{align}
\label{lian}
     \mathscr{H}^{n-1}(\partial B_{r_x}(x) \cap \Omega^1) \le C(n)P(\Omega; B_{r_x}(x)).
\end{align}
Since $\partial \Omega \cap \Omega^0 \subset \cup_{x \in \partial \Omega \cap \Omega^0} B_{r_x}(x)$, by Besicovitch covering theorem, we can choose $\xi(n)$ collection of balls $\cup_{i=1}^{\xi(n)} \cup_{B_{ij} \in \Lambda _i} B_{ij}$ from $\{B_{r_x}(x): x\in \partial \Omega \cap \Omega^0\}$ such that the union of these balls cover $\partial \Omega \cap \Omega^0$, and in each collection $\Lambda_i$ the balls $B_{ij}$ are disjoint. Let us relabel these balls as $B_j$. Therefore by \eqref{short}, \eqref{lian} and the choice of $B_j$ we have
\begin{align}
\label{smallnew}
    P(\cup_j B_j, \Omega^1) \lesssim_n \sum_{j} \mathscr{H}^{n-1}(\partial B_j \cap \Omega^1) \le C(n)\xi(n)P(\Omega, U)<\epsilon C(n)\xi(n).
\end{align}
From the above we can certainly choose $r_{B_j}<\epsilon$, where $r_{B_j}$ is the radius of the ball $B_j$, and hence again by \cite[Proposition 12.37]{Maggi} and the choice of $B_j$, we have 
\begin{align}
\label{smallvolume1}
    \sum_j |B_j \cap \Omega| \lesssim_n \epsilon \sum_j |\Omega \cap B_j|^{\frac{n-1}{n}} \lesssim_n \epsilon \sum_j P(\Omega, B_j) \lesssim_n  \epsilon\xi(n) P(\Omega, U)\lesssim_n \epsilon^2 \xi(n).
\end{align}
In view of \eqref{smallnew} and \eqref{smallvolume1}, so far we have chosen balls covering $\partial \Omega \cap \Omega^0$ such that the sum of perimeters of the balls inside $\Omega$ has order $\epsilon$, and the sum of the their volumes has order $\epsilon^2$. 

Next, by repeating the argument of the proof of \cite[Theorem 1.1]{Schmidt}, we can decompose $\partial^* \Omega$ into countably many disjoint flat pieces $R_1, R_2, \cdots$, and then each flat piece $R_i$ is contained in a countable union of flat cylinders $C_{i,j},\, j=1, 2, \cdots$ Moreover, for each $i$, the sum of the perimeters of $C_{i,j}, \, j=1,2,\cdots$ inside $\Omega$ is less than $(1+C(n)\epsilon)\mathscr{H}^{n-1}(R_i)$, which can be derived from the exact same argument as shown in \cite{Schmidt}. Hence we have
\begin{align}
\label{aaa0}
    P(\cup_{i,j}C_{i,j},\Omega) \le(1+C(n)\epsilon)P(\Omega).
\end{align} Using $r_{i,j},h_{i,j}$ as the same notations in the proof of \cite[Theorem 1.1]{Schmidt}, one can also see that
\begin{align}
\label{smallvolume2}
    \sum_j |C_{i,j}| \lesssim_n \sum_j h_{i,j}r_{i,j}^{n-1}\lesssim_n \epsilon^2 \left(\mathscr{H}^{n-1}(R_i)+2^{-i}\epsilon\right),
\end{align}hence \begin{align}
\label{aaa1}
    \sum_{i,j} |C_{i,j}| \lesssim_n \epsilon^2( P(\Omega)+\epsilon).
\end{align}
Slightly modifying the size of each flat cylinder $C_{i,j}$ does not affect the previous assertions, and since \begin{align}
\label{equivalence}
    \mathscr{H}^{n-1}(\partial \Omega \cap \Omega^1)<\infty,
\end{align}
by \cite[Proposition 2.16]{Maggi} we may assume that for each $C_{i,j}$, 
\begin{align}
\label{foliation2}
    \mathscr{H}^{n-1}\left(\partial C_{i,j} \cap (\partial \Omega \cap \Omega^1)\right)=0.
\end{align}
hence from \eqref{aaa0} and \eqref{foliation2} we have
\begin{align}
\label{aaa2}
    P(\cup_{i,j}C_{i,j},\Omega^1) \le(1+C(n)\epsilon)P(\Omega).
\end{align}
Next, since $$\partial \Omega \setminus \left((\partial \Omega \cap \Omega^0) \cup \partial^* \Omega\right)=(\partial^m \Omega \setminus \partial^* \Omega) \cup (\partial \Omega \cap \Omega^1)$$
which is $\mathscr{H}^{n-1}$-equivalent to $\partial\Omega \cap \Omega^1$ due to Theorem \ref{Fed}, we can also choose small balls $B_j'$, whose radii are less than $\epsilon$, to cover $\partial \Omega \setminus \left((\partial \Omega \cap \Omega^0) \cup \partial^* \Omega\right)$ such that
\begin{align}
    \label{ziwoguanchadakeng}
\sum_j \mathscr{H}^{n-1}(\partial B_j') \le C(n) \mathscr{H}^{n-1}(\partial \Omega \cap \Omega^1)+\epsilon.
\end{align}
In the above we have also used the equivalency of Hausdorff measure up to dimensional constant when the covering families in the definition of Hausdorff measure are balls instead of general sets.
Also from \eqref{ziwoguanchadakeng} and the choice of $B_j'$, we have
\begin{align}
\label{volumesmall}
    \sum_{j}|B_j'|\lesssim_n \epsilon
\end{align}

Now let $\mathcal{Q}_{\epsilon}$ be the union of all the balls $B_j$ covering $\partial \Omega \cap \Omega^0$, the flat cylinders $C_{ij}$ covering $\partial^* \Omega$ and the balls $B_j'$ covering $\partial \Omega \setminus \left((\partial \Omega \cap \Omega^0) \cup \partial^* \Omega\right)$ as we have chosen above. Since $\partial \Omega$ is compact, we can choose a finite number of the above covering sets whose union, which is denoted as $\mathcal{Q}'_{\epsilon}$, still covers $\partial \Omega$. Let $E_{\epsilon}=\Omega \setminus \mathcal{Q}'_{\epsilon}$. Then from \eqref{smallnew}, \eqref{smallvolume1}, \eqref{aaa1}, \eqref{aaa2}, \eqref{foliation2}, \eqref{ziwoguanchadakeng}, \eqref{volumesmall} and \cite[Theorem 16.3]{Maggi}, we have that $E_{\epsilon} \Subset \Omega$, $\lim_{\epsilon \rightarrow 0}|\Omega \setminus E_{\epsilon}|=0$ and that 
\begin{align*}
    P(E_{\epsilon})=P(\mathcal{Q}'_{\epsilon}, \Omega^1) &\le P(\cup_j B_j, \Omega^1)+P(\cup_{i,j}C_{i,j},\Omega)+\sum_j \mathscr{H}^{n-1}(\partial B_j')\\ & \le C(n)\epsilon+\left(1+C(n)\epsilon\right) P(\Omega)+C(n) \mathscr{H}^{n-1}(\partial \Omega \cap \Omega^1)+ \epsilon.
\end{align*} Hence we obtained the desired sequence $E_{\epsilon}$ as $\epsilon \rightarrow 0$, and this finishes the proof.
\end{proof}

Next, we prove Theorem \ref{nec}.

\begin{proof}[Proof of Theorem \ref{nec}]:
It suffices for us to prove that for any sequence $E_i \Subset \Omega$ such that $E_i \rightarrow \Omega$ in $L^1$, we can always find a subsequence of $E_i$ satisfying \eqref{wa}.

First if $\mathscr{H}^{n-1}(\partial \Omega \cap \Omega^1)=0$, then \eqref{wa} is immediate from lower semicontinuity of sets of finite perimeter. 

If $0<\mathscr{H}^{n-1}(\partial \Omega \cap \Omega^1)<\infty$, note that by \cite[Theorem 6.4]{Maggi} and Theorem \ref{Fed}, we can choose $X \subset \partial \Omega \cap \Omega^1$ such that $\mathscr{H}^{n-1}(X)=\mathscr{H}^{n-1}(\partial \Omega \cap \Omega^1)$ and that for any $x \in X$,
\begin{align}
\label{z1}
    \limsup_{r\rightarrow 0} \frac{P(\Omega; B_r(x))}{r^{n-1}}=0.
\end{align}
Also, by \cite[Theorem 2.3.2]{Evans}, we may also assume that for every $x \in X$, 
\begin{align}
\label{z2}
    \limsup_{r \rightarrow 0} \frac{\mathscr{H}^{n-1}(X \cap B_r(x)) }{\omega_{n-1} r^{n-1}} \ge \frac{1}{2^{n-1}}.
\end{align}
For any $\epsilon>0$,  by Egorov's Theorem, we can choose compact set $K \subset \partial \Omega \cap \Omega^1$ with $\mathscr{H}^{n-1}\left(X \setminus K\right)<\epsilon$ such that \eqref{z1} converges uniformly for $x \in K$. Hence there exists $\delta_1=\delta_1(\epsilon)>0$ such that for any $x \in K$ we have
\begin{align}
\label{liezi0}
    P(\Omega; B_r(x))<\epsilon r^{n-1}, \quad \forall \, 0<r<\delta_1
\end{align}
By \eqref{z2}, for each $x \in K$, we may choose $0<\delta_x<\delta_1$ with \begin{align}
\label{buchong}
    \delta_x^{n-1} \le \frac{2^{n-1}}{\omega_{n-1}} \mathscr{H}^{n-1}(\partial \Omega \cap \Omega^1 \cap B_{\delta_x}(x)).
\end{align}
Since $\{B_{\delta_x}(x)\}_{x \in K}$ cover $K$, by Besicovitch covering theorem, we can choose at most $\xi(n)$ collection of disjoint balls belonging to $\{B_{\delta_x}(x)\}_{x \in K}$ such that the union $U$ of these balls also contains $K$. We donote $U=\cup_{i=1}^{\xi(n)} \cup_{B_j \in \Lambda_i}B_j$, where $\Lambda_i$ denotes each collection of the disjoint balls. Hence it follows
\begin{align}
    P(\Omega; U) \le& \sum_{i=1}^{\xi(n)} \sum_{B_j \in \Lambda_i} P(\Omega; B_j), \quad \mbox{since $|\mu_{\Omega}|$ is an outer measure}\\
    \le& \epsilon \sum_{i=1}^{\xi(n)}\sum_{B_j \in \Lambda_i} r_{B_j}^{n-1}, \quad \mbox{by \eqref{liezi0},}\\
\label{cru}    \le & \epsilon \xi(n)\frac{2^{n-1}}{\omega_{n-1}}  \mathscr{H}^{n-1}(X), \quad \mbox{by \eqref{buchong}}.
\end{align}
In the above $r_{B_j}$ is denoted as the radius of the ball $B_j$. Since $K$ is compact, $U$ contains an $r_0$-neighborhood of $K$, where $0<r_0<\delta_1$.

Let $0<r_1<r_0$ and $K^{r_1}$ be the $r_1$-neighborhood of the compact set $K$, and thus \begin{align}
    \label{clarify}
\overline{K^{r_1}} \subset U.    
\end{align}

By lower semicontinuity of perimeter measure, we can choose a natural number $N>0$, such that for $i>N$,
\begin{align}
\label{yy}
    P(E_i; \mathbb{R}^n \setminus \overline{K^{r_1}}) > P(\Omega; \mathbb{R}^n \setminus \overline{K^{r_1}})-\epsilon.
\end{align}

By definition of Hausdorff measure, for any collection of balls $\{B_i\}$ covering $K$, as long as the radii of the balls are small, say less than some
\begin{align}
\label{qufa}
    0<\delta<r_1,
\end{align}then we have
\begin{align}
\label{liezi4}
    \sum_k \omega_{n-1} r_{B_k}^{n-1} \ge \mathscr{H}^{n-1}(K)-\epsilon.
\end{align}We may further choose $i=i(\delta)>N$ such that
\begin{align}
\label{liezi1}
    |\Omega \setminus E_i|\le \frac{\omega_n\delta^n}{3}.
\end{align}
Note that such $i$ eventually depend only on $\epsilon$.
Since 
\begin{align*}
    \lim_{r\rightarrow 0} \frac{|\Omega \cap B_r(x)|}{\omega_n r^n}=1, \quad \forall x \in K
\end{align*}
and $E_i \Subset \Omega$,  we have
\begin{align}
\label{dens1}
    \lim_{r\rightarrow 0} \frac{|(\Omega \setminus E_i) \cap B_r(x)|}{\omega_n r^n}=1, \quad \forall x \in K.
\end{align}
On the other hand, since $\Omega$ is bounded, we know
\begin{align}
\label{dens0}
    \lim_{r\rightarrow \infty} \frac{|(\Omega \setminus E_i) \cap B_r(x)|}{\omega_n r^n}=0, \quad \forall x \in K
\end{align}
Since for each $x \in K$ and $\mathscr{H}^{n-1}-a.e.$ $r>0$ we have \begin{align}
\label{foliation}
    \mathscr{H}^{n-1}\left(\partial B_{r}(x) \cap (\partial^* E_i \cup \partial^* \Omega)\right)=0,
\end{align}
by \eqref{dens1}, \eqref{dens0} and \eqref{foliation} we can find $r_x>0$ such that 
\begin{align}
\label{douqi}
    \mathscr{H}^{n-1}\left(\partial B_{r_x}(x) \cap (\partial^* E_i \cup \partial^* \Omega)\right)=0
\end{align}and 
\begin{align}
\label{liezi2}
    \frac{|(\Omega \setminus E_i) \cap B_{r_x}(x)|}{\omega_n r_x^n} \in (\frac{1}{3},\frac{1}{2}).
\end{align}
The choice of $r_x$ enables us to apply the relative isoperimetric inequality.

Note that \eqref{liezi1} and \eqref{liezi2} imply $r_x <\delta$, and since $\delta<\delta_1$ as our choice, we can apply the estimate \eqref{liezi0}. Indeed, we have
\begin{align*}
    r_x^{n-1} &\le C(n) \min\{|(\Omega \setminus E_i )\cap B_{r_x}(x)|, B_{r_x}(x)\setminus (\Omega \setminus E_i )|\}, \quad  \mbox{by \eqref{liezi2}} \\
    &\le C_1(n) P(\Omega \setminus E_i, B_{r_x}(x)), \quad \mbox{by \eqref{liezi2} and \cite[Proposition 12.37]{Maggi}}\\
    &= C_1(n)\left(P(E_i, B_{r_x}(x))+P(\Omega, B_{r_x}(x))\right),\quad \mbox{by \eqref{douqi}}\\
    &\le  C_1(n)P(E_i, B_{r_x}(x))+ C_1(n)\epsilon r_x^{n-1}, \quad \mbox{by \eqref{liezi0}}
\end{align*}
Hence
\begin{align}
    \label{liezi5}
P(E_i, B_{r_x}(x)) \ge (C_2(n)-\epsilon) r_x^{n-1}    
\end{align}
Since $K \subset \cup_{x \in A} B_{r_x}(x)$, by Besicovitch covering theorem we have
\begin{align}
\label{liezi6}
    K \subset G:=\cup_{i=1}^{\xi(n)} \cup_{B_j \in \Lambda'_i} B_j \subset K^{r_1},
\end{align}
where $\Lambda'_j$ is the collection of disjoint balls belonging to $\{B_{r_x}(x)\}_{x \in A}$. 
Therefore,
\begin{align}
    \xi(n) P(E_i, G) \ge & \sum_{i=1}^{\xi(n)} \sum_{B_j \in \Lambda_i'} P(E_i, B_j), \quad \mbox{since $\Lambda_i$ contains disjoint balls} \\
    \ge &  (C_2(n)-\epsilon) \sum_{i=1}^{\xi(n)}\sum_{B \in \cup_{i=1}^{\xi(n)} \Lambda_i'} r_{B}^{n-1}, \quad \mbox{by \eqref{liezi5}}\\
   \label{xishu} \ge &  \frac{C_2(n)-\epsilon}{\omega_{n-1}} (\mathscr{H}^{n-1}(K)-\epsilon), \quad \mbox{by \eqref{liezi4} and \eqref{liezi6}}.
\end{align}
Therefore, we have 
\begin{align}
    P(E_i) \ge& P(E_i, \mathbb{R}^n \setminus \overline{K^{r_1}}) + P(E_i, G), \quad \mbox{by \eqref{liezi6}}\\
    \ge & P(\Omega, \mathbb{R}^n \setminus \overline{K^{r_1}})-\epsilon+\frac{C_2(n)-\epsilon}{\xi(n)\omega_{n-1}} (\mathscr{H}^{n-1}(K)-\epsilon), \quad \mbox{by \eqref{yy} and \eqref{xishu}}\\
    \ge & P(\Omega)-P(\Omega,U)-\epsilon+\frac{C_2(n)-\epsilon}{\xi(n)\omega_{n-1}} (\mathscr{H}^{n-1}(K)-\epsilon), \quad \mbox{by \eqref{clarify}}\\
\label{wuqiong}    \ge & P(\Omega)-\epsilon\xi(n)\frac{2^{n-1}}{\omega_{n-1}} \mathscr{H}^{n-1}(X)-\epsilon+\frac{C_2(n)-\epsilon}{\xi(n)\omega_{n-1}} (\mathscr{H}^{n-1}(K)-\epsilon), \quad \mbox{by \eqref{cru}}\\
 \label{cont}   \ge & P(\Omega)+\frac{C_2(n)}{\xi(n)\omega_{n-1}} \mathscr{H}^{n-1}(\partial \Omega \cap \Omega^1)-C\epsilon, \quad \mbox{by the choice of $X$ and $K$}
\end{align}
Overall, the above estimates indicate that for each $\epsilon>0$, we can choose $E_i$ for some $i$ large which eventually only depends on $\epsilon$, such that \eqref{cont} holds. Letting $\epsilon \rightarrow 0$, we conclude \eqref{wa} for a subsequence of $E_i$. This proves the theorem for the case $0<\mathscr{H}^{n-1}(\partial \Omega \cap \Omega^1) <\infty$.

Last, if $\mathscr{H}^{n-1}(\partial \Omega \cap \Omega^1)=\infty$, then for any large number $M>0$, we can choose $X \subset \partial \Omega \cap \Omega^1$ such that $M<\mathscr{H}^{n-1}(X)<\infty$. Then follow the exact argument as above, by \eqref{wuqiong} and let $\epsilon \rightarrow 0$, we have
$$\liminf_{i \rightarrow \infty} P(E_i) \ge P(\Omega)+M.$$ Since $M$ is arbituary, \eqref{wa} is still true when $\mathscr{H}^{n-1}(\partial \Omega \cap \Omega^1)=\infty$.
\end{proof}

\section{Acknowledgement}
The third author is grateful to Professor Monica Torres for her introducing him to the study of interior approximation of sets of finite perimeter when he was studying at Purdue University.  This research is partially supported by NSF grants DMS-1601885 and DMS-1901914 and a Simons Foundation grant  Award 617072.

\end{document}